\documentclass[a4paper,12pt]{amsart}

\usepackage{amssymb}
\usepackage{amsmath}
\usepackage{tikz}
\usepackage[all,cmtip]{xy}
\newtheorem{thm}{Theorem}[section] 
\newtheorem{lem}[thm]{Lemma} 
 
\newtheorem{prop}[thm]{Proposition}
 
\theoremstyle{definition}
\newtheorem*{ack}{Acknowledgments}
\newtheorem{defi}[thm]{Definition} 
\newtheorem{rem}[thm]{Remark}

\title[]{An upper bound on the degree of singular vectors for $E(1,6)$}
\author{Lucia Bagnoli}
\address{Lucia Bagnoli, Dipartimento di matematica, Universit\`a di Bologna, Piazza di Porta San Donato 5, 40126 Bologna, Italy}

\email{luciabagnoli93@gmail.com}
\DeclareMathOperator{\Res}{Res}

\DeclareMathOperator{\Ind}{Ind}
\DeclareMathOperator{\Cur}{Cur}
\DeclareMathOperator{\Lie}{Lie}
\DeclareMathOperator{\Sing}{Sing}

\DeclareMathOperator{\End}{End}
\DeclareMathOperator{\I}{I}

\DeclareMathOperator{\sgn}{sgn}

\newcommand{\C}{\mathbb{C}}
\newcommand{\Z}{\mathbb{Z}}

\newcommand{\de}{\partial}
\newcommand{\so}{\mathfrak{so}}
\DeclareMathOperator{\spann}{span}
\newcommand{\inlinewedge}{\textrm{\raisebox{0.6mm}{\footnotesize $\bigwedge$}}}
\newcommand{\displaywedge}{\textrm{\raisebox{0.6mm}{\tiny $\bigwedge$}}}
\newcommand{\g}{\mathfrak {g}}
\DeclareMathOperator{\degr}{deg}

\newcommand*{\bigchi}{\mbox{\Large$\chi$}}
\allowdisplaybreaks

\begin{document}
\maketitle
\begin{abstract}
The aim of this work is to prove a technical result, that had been stated by Boyallian, Kac and Liberati \cite{ck6}, on the degree of singular vectors of finite Verma modules over the exceptional Lie superalgebra $E(1,6)$ that is isomorphic to the annihilation superalgebra associated with the conformal superalgebra $CK_{6}$. 
\end{abstract}
\section{Introduction}
Finite simple conformal superalgebras were completely classified in \cite{fattorikac} and consist in the list: $\Cur \mathfrak{g}$, where $\mathfrak{g}$ is a simple finite$-$dimensional Lie superalgebra, $W_{n} (n\geq 0)$, $S_{n,b}$, $\tilde{S}_{n}$ $(n\geq 2, \, b \in \mathbb{C})$, $K_{n} (n\geq 0, \, n \neq 4)$, $K'_{4}$, $CK_{6} $. The finite irreducible modules over the conformal superalgebras $\Cur \mathfrak{g}$, $K_{0}$, $K_{1}$ were studied in \cite{chengkac}. The classification of all finite irreducible modules over the conformal superalgebras $S_{2,0}$, $K_{N}$, for $N  =2,3,4$ was obtained in \cite{chenglam}. Boyallian, Kac, Liberati and Rudakov classified all finite irreducible modules over the conformal superalgebras of type $W$ and $S$ in \cite{bklr}; Boyallian, Kac and Liberati classified all finite irreducible modules over the conformal superalgebras of type $K_{N}$ for $N \geq 4$ in \cite{kac1}.  All finite irreducible modules over the conformal superalgebra $K'_{4}$ were classified in \cite{bagnoli}. Finally a classification of all finite irreducible modules over the conformal superalgebra $CK_{6}$ was obtained in \cite{ck6} and \cite{zm} with different approaches. \par
In \cite{ck6} the classification of all finite irreducible modules over the conformal superalgebra $CK_{6}$ is obtained by their correspondence with irreducible finite conformal modules over the annihilation superalgebra $\g:=\mathcal{A}(CK_{6})$ associated with $CK_{6}$. The annihilation superalgebra $\mathcal{A}(CK_{6})$ is isomorphic to the exceptional Lie superalgebra $E(1,6)$ (see \cite{new6},\cite{chengkac2},\cite{Shche},\cite{kac98}). In \cite{ck6}, in order to obtain this classification, the authors classify all highest weight singular vectors of finite Verma modules, i.e. induced modules $\Ind(F)=U(\g) \otimes _{U(\g_{\geq 0})} F$, where $F$ is a finite$-$dimensional irreducible $\g_{\geq 0}$-module \cite{kacrudakov,chenglam}. In \cite{ck6} the classification of highest weight singular vectors is based on a technical lemma, whose proof is missing (Lemma $4.4$  in \cite{ck6}), that provides an upper bound on the degree of singular vectors for $E(1,6)$. \par
The aim of this paper is to prove that technical lemma stated in \cite{ck6}. The proof of this lemma completes the classification of singular vectors for $E(1,6)$ given in \cite{ck6}. \par
The paper is organized as follows. In section 2 we recall some notions on conformal superalgebras. In section 3 we recall the definition of the conformal superalgebra $CK_6$ and some of its properties. Finally, in section 4 we prove the bound on the degree of singular vectors for $E(1,6)$.
\section{Preliminaries on conformal superalgebras}
We recall some notions on conformal superalgebras. For further details see \cite[Chapter 2]{kac1vertex}, \cite{dandrea}, \cite{bklr}, \cite{kac1}.\\
Let $\g$ be a Lie superalgebra; a formal distribution with coefficients in $\g$, or equivalently a $\g-$valued formal distribution, in the indeterminate $z$ is an expression of the following form:
\begin{align*}
a(z)=\sum_{n \in \Z}a_{n}z^{-n-1},  
\end{align*}
with $a_{n} \in \g$ for every $n \in \Z$. We denote the vector space of formal distributions with coefficients in $\g$ in the indeterminate $z$ by $\g[[z,z^{-1}]]$. We denote by $\Res(a(z))=a_{0}$ the coefficient of $z^{-1}$ of $a(z)$. The vector space $\g[[z,z^{-1}]]$ has a natural structure of $\C[\partial_{z}]-$module. We define for all $a(z) \in \g[[z,z^{-1}]]$ its derivative:
\begin{align*}
\partial_{z}a(z)=\sum_{n \in \Z}(-n-1)a_{n}z^{-n-2}.
\end{align*} 
A formal distribution with coefficients in $\g$ in the indeterminates $z$ and $w$ is an expression of the following form:
\begin{align*}
a(z,w)=\sum_{m,n \in \Z}a_{m,n}z^{-m-1}w^{-n-1},  
\end{align*}
with $a_{m,n} \in \g$ for every $m,n \in \Z$. We denote the vector space of formal distributions with coefficients in $\g$ in the indeterminates $z$ and $w$ by $\g[[z,z^{-1},w,w^{-1}]]$.
Given two formal distributions $a(z) \in \g[[z,z^{-1}]]$ and $b(w) \in \g[[w,w^{-1}]]$, we define the commutator $[a(z),b(w)]$:
\begin{align*}
[a(z),b(w)]=\bigg[\sum_{n \in \Z}a_{n}z^{-n-1} ,\sum_{m \in \Z}b_{m}w^{-m-1}  \bigg]=\sum_{m,n \in \Z}[a_{n},b_{m}] z^{-n-1}  w^{-m-1} .
\end{align*}
\begin{defi}
Two formal distributions $a(z),b(z) \in \g[[z,z^{-1}]]$ are called local if:
\begin{align*}
(z-w)^{N}[a(z),b(w)]=0 \ for \ some \ N \gg 0.
\end{align*}
\end{defi}
We call $\delta-$function the following formal distribution in the indeterminates $z$ and $w$:
\begin{align*}
\delta(z-w)=z^{-1}\sum_{n \in \Z} \left( \frac{w}{z} \right)^{n}.
\end{align*}
See Corollary $2.2$ in \cite{kac1vertex} for the following equivalent condition of locality.
\begin{prop}
Two formal distributions $a(z),b(z) \in \g[[z,z^{-1}]]$ are local if and only if $[a(z),b(w)]$ can be expressed as a finite sum of the form:
\begin{align*}
[a(z),b(w)]=\sum_{j}(a(w)_{(j)}b(w))\frac{\partial_{w}^{j}}{j!} \delta(z-w),
\end{align*}
where the coefficients $(a(w)_{(j)}b(w)):=\Res_{z}(z-w)^{j}[a(z),b(w)]$ are formal distributions in the indeterminate $w$.
\end{prop}
\begin{defi}[Formal Distribution Superalgebra]
Let $\g$ be a Lie superalgebra and $\mathcal{F}$ a family of mutually local $\g-$valued formal distributions in the indeterminate $z$. The pair $(\g,\mathcal{F})$ is called a \textit{formal distribution superalgebra} if the coefficients of all formal distributions in $\mathcal{F}$ span $\g$.
\end{defi}
We define the $\lambda-$bracket between two formal distributions $a(z),b(z) \in \g[[z,z^{-1}]] $ as the generating series of the $(a(z)_{(j)}b(z))$'s:
\begin{align}
\label{bracketformal}
[a(z)_{\lambda}b(z)]=\sum_{j \geq 0} \frac{\lambda^{j}}{j!}(a(z)_{(j)}b(z)).
\end{align}
\begin{defi}[Conformal superalgebra]
A \textit{conformal superalgebra} $R$ is a left $\Z_{2}-$graded $\C[\partial]-$module endowed with a $\C-$linear map, called $\lambda-$bracket, $R \otimes R \rightarrow \C[\lambda]\otimes R$, $a \otimes b \mapsto [a_{\lambda}b]$, that satisfies the following properties for all $a,b,c \in R$:
\begin{align*}
(i)&\,\, conformal \,\, sesquilinearity: &&[\partial a_{\lambda}b]=-\lambda [a_{\lambda}b], \quad  [a_{\lambda} \partial b]=(\lambda+\partial)[a_{\lambda}b]; \\
(ii)&\,\, skew-symmetry:             &&[a_{\lambda}b]=-(-1)^{p(a)p(b)}[b_{-\lambda-\partial}a] ;    \\
(iii)&\,\, Jacobi \,\, identity:           &&[a_{\lambda}[b_{\mu}c]]=[[a_{\lambda}b]_{\lambda+\mu}c]+(-1)^{p(a)p(b)}[b_{\mu}[a_{\lambda}c]];
\end{align*}
where $p(a)$ denotes the parity of the element $a \in R$ and $p(\partial a)=p(a)$ for all $a \in R$.
\end{defi}
We call $n-$products the coefficients $(a_{(n)}b)$ that appear in $[a_{\lambda}b]=\sum_{n\geq 0} \frac{\lambda^{n}}{n!}(a_{(n)}b)$  and give an equivalent definition of conformal superalgebra.
\begin{defi}[Conformal superalgebra]
\label{definizionesuperalgebraconforme}
A \textit{conformal superalgebra} $R$ is a left $\Z_{2}-$graded $\C[\partial]-$module endowed with a $\C-$bilinear product $(a_{(n)}b): R\otimes R\rightarrow R$, defined for every $n \geq 0$, that satisfies the following properties for all $a,b,c \in R$, $m,n \geq 0$:
\begin{itemize}
  \item[(i)] $p(\partial a)=p( a)$;
  \item[(ii)] $(a_{(n)}b)=0, \,\, for \,\, n \gg 0$;
	\item[(iii)] $({\partial a}_{(0)}b)=0$ and $({\partial a}_{(n+1)}b)=-(n+1) (a_{(n)}b)$;
	\item[(iv)] $(a_{(n)}b)=-(-1)^{p(a)p(b)}\sum_{j \geq 0}(-1)^{j+n} \frac{\partial^{j}}{j!}(b_{(n+j)}a)$;
	\item[(v)] $(a_{(m)}(b_{(n)}c))=\sum^{m}_{j=0}\binom{m}{j}((a_{(j)}b)_{(m+n-j)}c)+(-1)^{p(a)p(b)}(b_{(n)}(a_{(m)}c))$.
\end{itemize}
\end{defi}
Using (iii) and (iv) in Definition \ref{definizionesuperalgebraconforme} it is easy to show that for all $a,b \in R$, $n \geq 0$:
\begin{equation*}
(a_{(n)}\partial b)=\partial (a_{(n)} b)+n (a_{(n-1)}b).
\end{equation*}
Due to this relation and (iii) in Definition \ref{definizionesuperalgebraconforme}, the map $\partial: R\rightarrow R$, $a \mapsto \partial a$ is a derivation with respect to the $n-$products.
\begin{rem}
Let $(\g ,\mathcal{F})$ be a formal distribution superalgebra, endowed with $\lambda-$bracket (\ref{bracketformal}). The elements of $\mathcal{F}$ satisfy sesquilinearity, skew-symmetry and Jacobi identity with $\partial=\partial_{z}$; for a proof see Proposition 2.3 in \cite{kac1vertex}.
\end{rem}
We say that a conformal superalgebra $R$ is \textit{finite} if it is finitely generated as a $\C[\partial]-$module. 
An \textit{ideal} $I$ of $R$ is a $\C[\partial]-$submodule of $R$ such that $(a_{(n)}b)\in I$ for every $a \in R$, $b \in I$, $n \geq 0$. A conformal superalgebra $R$ is \textit{simple} if it has no non-trivial ideals and the $\lambda-$bracket is not identically zero. We denote by $R'$ the \textit{derived subalgebra} of $R$, i.e. the $\C-$span of all $n-$products.
\begin{defi}
A module $M$ over a conformal superalgebra $R$ is a left $\Z_{2}-$graded $\C[\partial]-$module endowed with $\C-$linear maps $R \rightarrow \End_{\C} M$, $a\mapsto a_{(n)}$, defined for every $n \geq 0$, that satisfy the following properties for all $a,b \in R$, $v \in M$, $m,n \geq 0$:
\begin{enumerate}
  \item[(i)] $a_{(n)}v=0 \,\, for \,\, n \gg 0$;
	\item[(ii)] $(\partial a)_{(n)}v=[\partial,a_{(n)}]v=-n a_{(n-1)}v $;
	\item[(iii)] $[a_{(m)},b_{(n)}]v=\sum_{j=0}^{m} \binom{m}{j}(a_{(j)}b)_{(m+n-j)}v$.
\end{enumerate}
\end{defi}
Given a module $M$ over a conformal superalgebra $R$, we define for all $a \in R$ and $v \in M$:
\begin{align*}
a_{\lambda}v=\sum_{n\geq 0} \frac{\lambda^{n}}{n!} a_{(n)}v.
\end{align*}
A module $M$ is called \textit{finite} if it is a finitely generated $\C[\partial]-$module.\\
We can construct a conformal superalgebra starting from a formal distribution superalgebra $(\g,\mathcal{F})$. Let $\mathcal{\overline{F}}$  be the closure of $\mathcal{F}$ under all the $n-$products, $\partial_{z}$ and linear combinations. By  Dong's Lemma, $\mathcal{\overline{F}}$ is still a family of mutually local distributions (see \cite{kac1vertex}). It turns out that $\mathcal{\overline{F}}$ is a conformal superalgebra. We will refer to it as the conformal superalgebra associated with $(\g,\mathcal{F})$.\\
Let us recall the construction of the annihilation superalgebra associated with a conformal superalgebra $R$.
Let $\widetilde{R}=R[y,y^{-1}]$, set $p(y)=0$ and $\widetilde{\partial}=\partial+\partial_{y}$. We define the following $n-$products on $\widetilde{R}$, for all $a,b \in R$, $f,g \in \C[y,y^{-1}]$, $n\geq 0$:
\begin{align*}
(af_{(n)}bg)=\sum_{j \in \Z_{+}}(a_{(n+j)}b) \Big(\frac{\partial_{y}^{j}}{j!}f \Big)g .
\end{align*}
In particular if $f=y^{m}$ and $g=y^{k}$ we have for all $n \geq 0$:
\begin{align*}
({ay^{m}}_{ (n)}by^{k})=\sum_{j \in \Z_{+}}\binom{m}{j}(a_{(n+j)}b)y^{m+k-j}.
\end{align*}
We observe that $\widetilde{\partial}\widetilde{R}$ is a two sided ideal of $\widetilde{R}$ with respect to the $0-$product. The quotient $\Lie R:=\widetilde{R}/ \widetilde{\partial}\widetilde{R}$ has a structure of Lie superalgebra with the bracket induced by the $0-$product, i.e. for all $a,b \in R$, $f,g \in \C[y,y^{-1}]$: 
\begin{align}
\label{bracketannihilation}
[af,bg]=\sum_{j \in \Z_{+}}(a_{( j)}b)\Big(\frac{\partial_{y}^{j}}{j!}f \Big)g .
\end{align}
\begin{defi}
The annihilation superalgebra $\mathcal{A}(R)$ of a conformal superalgebra $R$ is the subalgebra of $\Lie R$ spanned by all elements $ay^{n}$ with $n\geq 0$ and $a\in R$. \\
The extended annihilation superalgebra $\mathcal{A}(R)^{e}$ of a conformal superalgebra $R$ is the Lie superalgebra $\C \partial \ltimes \mathcal{A}(R)$. The semidirect sum $\C \partial \ltimes \mathcal{A}(R)$ is the vector space $\C \partial \oplus \mathcal{A}(R)$ endowed with the structure of Lie superalgebra determined by the bracket:
\begin{align*}
[\partial,ay^{m}]=-\partial_{y}(ay^{m})=-m a y^{m-1},
\end{align*}
for all $a \in R$ and the fact that $\C \partial$, $\mathcal{A}(R)$ are Lie subalgebras.
\end{defi}
For all $a \in R$ we consider the following formal power series in $\mathcal{A}(R)[[\lambda]]$:
\begin{align}
\label{powerseries}
a_{\lambda}=\sum_{n \geq 0}\frac{\lambda^{n}}{n!}ay^{n}.
\end{align}
For all $a,b \in R$, we have: $[a_{\lambda},b_{\mu}]=[a_{\lambda}b]_{\lambda+\mu}$ and $(\partial a)_{\lambda}=-\lambda a_{\lambda}$ (for a proof see \cite{cantacasellikac}).  
\begin{prop}[\cite{chengkac}]
\label{propcorrispmoduli}
Let $R$ be a conformal superalgebra.
	If $M$ is an $R$-module then $M$ has a natural structure of $\mathcal A(R)^e$-module, where the action of $ay^n$ on $M$ is uniquely determined by $a_\lambda v=\sum_{n \geq 0}\frac{\lambda^{n}}{n!}ay^{n}.v$ for all $v\in V$. Viceversa if $M$ is a $\mathcal A(R)^e$-module such that for all $a\in R$, $v\in M$ we have $ay^n.v=0$ for $n\gg0$, then $M$ is also an $R$-module by letting $a_\lambda v=\sum_{n}\frac{\lambda^{n}}{n!}ay^{n}.v$.
\end{prop}
Proposition \ref{propcorrispmoduli} reduces the study of modules over a conformal superalgebra $R$ to the study of a class of modules over its (extended) annihilation superalgebra.
The following proposition states that, under certain hypotheses, it is sufficient to consider the annihilation superalgebra. We recall that, given a $\Z-$graded Lie superalgebra $\g=\oplus_{i \in \Z}\g_{i}$, we say that $\g$ has finite depth $d\geq0$ if $\g_{-d}\neq 0$ and $\g_{i}=0$ for all $i<-d$.
\begin{prop}[\cite{kac1},\cite{chenglam}]
\label{keythmannihi}
Let $\g$ be the annihilation superalgebra of a conformal superalgebra $R$. Assume that $\g$ satisfies the following conditions:
\begin{description}
	\item[L1] $\g$ is $\Z-$graded with finite depth $d$;
	\item[L2] There exists an element whose centralizer in $\g$ is contained in $\g_{0}$;
	\item[L3] There exists an element $\Theta \in \g_{-d}$ such that $\g_{i-d}=[\Theta,\g_{i}]$, for all $i\geq 0$.
\end{description}
 Finite modules over $R$ are the same as modules $V$ over $\g$, called \textit{finite conformal}, that satisfy the following properties:
\begin{enumerate}
	\item for every $v \in V$, there exists $j_{0} \in \Z$, $j_{0}\geq -d$, such that $\g_{j}.v=0$ when $j\geq j_{0}$;
	\item $V$ is finitely generated as a $\C[\Theta]-$module.
\end{enumerate}
\end{prop}
\begin{rem}
\label{gradingelement}
We point out that condition \textbf{L2} is automatically satisfied when $\g$ contains a \textit{grading element}, i.e. an element $t \in \g$ such that $[t,b]=\degr (b) b$ for all $b \in \g$.
\end{rem}
Let $\g=\oplus_{i \in \Z} \g_{i}$ be a $\Z-$graded Lie superalgebra. We will use the notation $\g_{>0}=\oplus_{i>0}\g_{i}$, $\g_{<0}=\oplus_{i<0}\g_{i}$ and $\g_{\geq 0}=\oplus_{i\geq 0}\g_{i}$. We denote by $U(\g) $ the universal enveloping algebra of $\g$.
\begin{defi}
Let $F$ be a $\g_{\geq 0}-$module. The \textit{generalized Verma module} associated with $F$ is the $\g-$module $\Ind (F)$ defined by:
\begin{equation*}
\Ind (F):= \Ind ^{\g}_{\g_{\geq 0}} (F)=U(\g) \otimes _{U(\g_{\geq 0})} F.
\end{equation*}
\end{defi}
If $F$ is a finite$-$dimensional irreducible $\g_{\geq 0}-$module we will say that $\Ind(F)$ is a \textit{finite Verma module}. We will identify $\Ind (F)$ with $U(\g_{<0}) \otimes  F$ as vector spaces via the Poincar\'e$-$Birkhoff$-$Witt Theorem. The $\Z-$grading of $\g$ induces a $\Z-$grading on $U(\g_{<0})$ and $\Ind (F)$. We will invert the sign of the degree, so that we have a $\Z_{\geq 0}-$grading on $U(\g_{<0})$ and $\Ind (F)$. We will say that an element $v \in U(\g_{<0})_{k}$ is homogeneous of degree $k$. Analogously an element $m \in U(\g_{<0})_{k}  \otimes  F$ is homogeneous of degree $k$. For a proof of the following proposition see \cite{bagnoli}.
\begin{prop}
Let $\g=\oplus_{i \in \Z} \g_{i}$ be a $\Z-$graded Lie superalgebra. If $F$ is an irreducible finite$-$dimensional $\mathfrak{g}_{\geq 0}-$module, then $\Ind (F)$ has a unique maximal submodule. We denote by $\I (F)$ the quotient of $\Ind (F)$ by the unique maximal submodule.
\end{prop}
\begin{defi}
Given a $\g-$module $V$, we call \textit{singular vectors} the elements of:
\begin{align*}
\Sing (V) =\left\{v \in V \,\, | \, \, \g_{>0}.v=0\right\}.
\end{align*}
Homogeneous components of singular vectors are still singular vectors so we often assume that singular vectors are homogeneous without loss of generality.
If $V=\Ind (F)$, for a $\g_{\geq 0}-$module $F$, we will call \textit{trivial singular vectors} the elements of $\Sing (V) $ of degree 0 and \textit{nontrivial singular vectors} the nonzero elements of $\Sing (V) $ of positive degree.
\end{defi}
\begin{thm}[\cite{kacrudakov},\cite{chenglam}]
\label{keythmsingular}
Let $\g$ be a Lie superalgebra that satisfies \textbf{L1}, \textbf{L2}, \textbf{L3}, then:
\begin{enumerate}
\item[(i)] if $F$ is an irreducible finite$-$dimensional $\mathfrak{g}_{\geq 0}-$module, then $\mathfrak{g}_{> 0}$ acts trivially on it;
	\item[(ii)] the map $F \mapsto \I (F)$ is a bijective map between irreducible finite$-$dimensional $\mathfrak{g}_{ 0}-$modules and irreducible finite conformal $\mathfrak{g}-$modules;
	\item[(iii)] the $\mathfrak{g}-$module $\Ind (F)$ is irreducible if and only if the $\mathfrak{g}_{0}-$module $F$ is irreducible and $\Ind (F)$ has no nontrivial singular vectors.
	\end{enumerate}
\end{thm}
\section{The conformal superalgebra $CK_{6}$}
In this section we recall the definition and some properties of the conformal superalgebra $CK_{6}$ from \cite{ck6}.
Let $\inlinewedge(N)$ be the Grassmann superalgebra in the $N$ odd indeterminates $\xi_{1},...,\xi_{N}$. Let $t$ be an even indeterminate and $\inlinewedge (1,N)=\C[t,t^{-1}] \otimes \inlinewedge(N)$. We consider the Lie superalgebra of derivations of $\inlinewedge (1,N)$:
\begin{equation*}
W(1,N)=\bigg\{ D=a \partial_{t}+\sum ^{N}_{i=1} a_{i} \partial_{i} \,\, | \,\, a,a_{i} \in \displaywedge (1,N)\bigg\},
\end{equation*}
where $\partial_{t}=\frac{\partial}{\partial{t}}$ and $\partial_{i} =\frac{\partial}{\partial{\xi_{i}}}$ for every $i \in \left\{1,...,N \right\}$.\\
Let us consider the contact form $\omega = dt-\sum_{i=1}^{N}\xi_{i} d\xi_{i} $. The contact Lie superalgebra $K(1,N)$ is defined by:
\begin{equation*}
K(1,N)=\left\{D \in W(1,N) \,\, | \,\, D\omega=f_{D}\omega \,\, for \,\, some \,\, f_{D} \in \displaywedge (1,N)\right\}.
\end{equation*}
Analogously, let $\inlinewedge  (1,N)_{+}=\C[t] \otimes \inlinewedge (N)$. We define the Lie superalgebra $W(1,N)_{+}$ (resp. $K(1,N)_{+}$) similarly to $W(1,N)$ (resp. $K(1,N)$) using $\inlinewedge  (1,N)_{+}$ instead of $\inlinewedge  (1,N)$.
We can define on $\inlinewedge (1,N)$ a Lie superalgebra structure as follows. For all $f,g \in \inlinewedge(1,N)$ we let:
\begin{equation}
\label{bracketlie}
[f,g]=\Big(2f-\sum_{i=1}^{N} \xi_{i}  \partial_{i} f \Big)(\partial_{t}{g})-(\partial_{t}{f})\Big(2g-\sum_{i=1}^{N} \xi_{i} \partial_{i} g\Big)+(-1)^{p(f)}\Big(\sum_{i=1}^{N} \partial_{i}  f \partial_{i}  g \Big).
\end{equation}
We recall that $K(1,N) \cong \inlinewedge(1,N)$ as Lie superalgebras via the following map (see \cite{{chengkac2}}):
\begin{gather*}
\displaywedge(1,N) \longrightarrow  K(1,N) \\
f \longmapsto 2f \partial_{t}+(-1)^{p(f)} \sum_{i=1}^{N} (\xi_{i} \partial_{t} f+ \partial_{i}f )(\xi_{i} \partial_{t} + \partial_{i}).
\end{gather*}
We will always identify elements of $K(1,N) $ with elements of $\inlinewedge(1,N)$ and we will omit the symbol $\wedge$ between the $\xi_{i}$'s. We consider on $K(1,N)$ the standard grading, i.e. for every $t^{m} \xi_{i_{1}} \cdots \xi_{i_{s}} \in K(1,N)$ we have $\degr(t^{m} \xi_{i_{1}} \cdots \xi_{i_{s}})=2m+s-2$.\\
We consider the following family of formal distributions:
\begin{equation*}
\mathcal{F}=\bigg\{A(z):= \sum_{m \in \Z}(At^{m})z^{-m-1}=A \delta(t-z), \,\, \forall    A \in \displaywedge(N)\bigg\}. 
\end{equation*}
The pair ($K(1,N)$,$\mathcal{F}$) is a formal distribution superalgebra and the conformal superalgebra $\mathcal{\overline{F}}$ can be identified with $K_{N}:=\C[\partial] \otimes \inlinewedge(N)$ (for a proof see \cite{bagnoli}). We will refer to it as the conformal superalgebra of type $K$.\\
On $K_{N}$ the $\lambda-$bracket for $f,g \in \inlinewedge(N)$, $f= \xi_{i_{1}} \cdots \xi_{i_{r}}$ and $g= \xi_{j_{1}} \cdots \xi_{j_{s}}$, is given by (see \cite{kac1},\cite{fattorikac}):
\begin{align*}
[f_{\lambda}g]=(r-2)\partial(fg)+(-1)^{r}\sum^{N}_{i=1}(\partial_{i}f)(\partial_{i}g)+\lambda(r+s-4)fg.
\end{align*}
The associated annihilation superalgebra is (see \cite{kac1},\cite{fattorikac}):
\begin{align*}
\mathcal{A}(K_{N})=K(1,N)_{+}.
\end{align*}
We adopt the following notation: we denote by $\mathcal I$ the set of finite sequences of elements in $\{1,\ldots,N\}$; we will write $I=i_1\cdots i_r$ instead of $I=(i_1,\ldots,i_r)$. Given $I=i_1\cdots i_r$ and $J=j_1\cdots j_s$, we will denote $i_1\cdots i_rj_1\cdots j_s$ by $IJ$; if $I=i_1 \cdots i_r\in \mathcal I$  we let $\xi_{I}=\xi_{i_1}\cdots \xi_{i_r}$ and $|\xi_{I}|=|I|=r$. We denote by $\mathcal{I}_{\neq}$ the subset of $\mathcal{I}$ of sequences with distinct entries and by $\mathcal{I}_{<}$ the subset of $\mathcal{I}_{\neq}$ of increasingly ordered sequences. We focus on $N=6$. We let $\xi_{*}=\xi_{123456}$.

Following \cite{ck6}, for $\xi_{I} \in \inlinewedge(6)$ we define the \textit{modified Hodge dual} $\xi^{*}_{I}$ to be the unique monomial such that $\xi_{I}\xi^{*}_{I}=\xi_{*}$. We extend the definition of modified Hodge dual to elements $\sum_{k,I} \alpha_{k,I}t^{k}\xi_{I}\in \inlinewedge(1,6)_{+}$ letting $(\sum_{k,I} \alpha_{k,I}t^{k}\xi_{I})^{*}=\sum_{k,I} \alpha_{k,I}t^{k}\xi_{I}^{*}$.\\
The conformal superalgebra $CK_{6}$ is the subalgebra of $K_{6}$ defined by (see construction in \cite{new6}):
\begin{align*}
CK_{6}=\C[\partial]-\spann \left\{\xi_{L}-i(-1)^{\frac{|L|(|L|+1)}{2}}(-\partial)^{3-|L|}\xi_{L}^{*}: \, L \in \mathcal{I}_{\neq}, 0 \leq |L| \leq 3   \right\}.
\end{align*}
We introduce the linear operator $A: K(1,6)_{+}  \longrightarrow K(1,6)_{+}$:
\begin{align*}
A(t^{k}\xi_{L})=(-1)^{\frac{|L|(|L|+1)}{2}}\left(\frac{d}{dt}\right)^{3-|L|}(t^{k}\xi_{L})^{*},
\end{align*}
where $\left(\frac{d}{dt}\right)^{-1}$ indicates integration with respect to $t$ (i.e. it sends $t^{k}$ to $t^{k+1}/(k+1)$) and $A$ is extended by linearity (cf. Remark 5.3.2 in \cite{chengkac2}).
The annihilation superalgebra associated with $CK_{6}$, that we will denote by $\g$, is the subalgebra of $K(1,6)_{+}$ given by the image of $Id-iA$; it is isomorphic to the exceptional Lie superalgebra $E(1,6)$ (see \cite{chengkac2},\cite{new6},\cite{Shche},\cite{kac98}). The bracket on $\g$ is given by \eqref{bracketlie}.
\begin{rem}
\label{costruzioneannihi}
We point out that, $\g$ is in bijective correspondence with the span of elements $(Id-iA)(t^{k}\xi_{L})$ with  $L \in \mathcal{I}_{\neq}$, $|L| \leq 3$, $k\geq 0$.
 Indeed for $L \in \mathcal{I}_{\neq}$, with $|L| > 3$:
\begin{align*}
(Id-iA)(t^{k}\xi_{L})=(Id-iA)\Big(-i(-1)^{\frac{|L|(|L|+1)}{2}}\frac{t^{k+|L|-3}}{k(k+1) \cdots (k+|L|-3)}\xi_{L}^{*}\Big).
\end{align*}
\end{rem}
The map $A$ preserves the $\Z-$grading, then $\g$ inherits the $\Z-$grading. The homogeneous components of non$-$positive degree of $\g$ and $K(1,6)_{+}$ coincide and are:
\begin{align*}
&\g_{-2}=\langle 1 \rangle ,\\
&\g_{-1}=\langle \xi_{1},\xi_{2},\xi_{3},\xi_{4},\xi_{5},\xi_{6} \rangle , \\
&\g_{0}=\langle t,\xi_{ij}: \,\, 1\leq i,j\leq 6 \rangle .
\end{align*}
The annihilation superalgebra $\g$ satisfies \textbf{L1}, \textbf{L2}, \textbf{L3}: \textbf{L1} is straightforward; \textbf{L2} follows by Remark \ref{gradingelement} since $t$ is a grading element for $\g$; \textbf{L3} follows from the choice $\Theta:=-1/2 \in \g_{-2}$.
Let us focus on $\g_{0}=\langle t, \xi_{ij}: \quad 1 \leq i<j \leq 6 \rangle \cong \C t \oplus \mathfrak{so}(6) $, where $\mathfrak{so}(6)$ is the Lie algebra of $6 \times 6$ skew$-$symmetric matrices and $\xi_{ij}\in \g_{0} $ corresponds to $E_{j,i}-E_{i,j} \in \mathfrak{so}(6)$. 
 We recall the following notation from \cite{ck6}. We choose as basis of a Cartan subalgebra $\mathfrak{h}$ of $ \mathfrak{so}(6)$ the elements:
\begin{align*}
H_{1}=-i\xi_{12}, \, H_{2}=-i\xi_{34}, \, H_{3}=-i\xi_{56}.
\end{align*}
Let $\varepsilon_{j}  \in \mathfrak{h}^{*}$ such that $\varepsilon_{j}(H_{k})=\delta_{j,k}$.
The roots are $\Delta=\left\{\pm\varepsilon_{l}\pm\varepsilon_{j}:\, 1 \leq l<j \leq 3 \right\}$, the positive roots are $\Delta^{+}=\left\{\varepsilon_{l}\pm\varepsilon_{j}:\, 1 \leq l<j \leq 3 \right\}$ and the simple roots are $\Pi=\{\varepsilon_{1}-\varepsilon_{2}, \varepsilon_{2}-\varepsilon_{3}, \varepsilon_{2}+\varepsilon_{3} \}$.
The root decomposition is $\mathfrak{so}(6)=\mathfrak{h} \oplus \left(\oplus_{\alpha \in \Delta} \g_{\alpha}\right)$, where $\g_{\alpha}=\C E_{\alpha}$ and the $E_{\alpha}$'s are, for $1 \leq l<j \leq 3$:
\begin{align*}
E_{\varepsilon_{l}-\varepsilon_{j}}&=-\xi_{2l-1, 2j-1}-\xi_{2l, 2j}-i\xi_{2l-1,2j}+i\xi_{2l, 2j-1},\\
E_{\varepsilon_{l}+\varepsilon_{j}}&=-\xi_{2l-1, 2j-1}+\xi_{2l, 2j}+i\xi_{2l-1,2j}+i\xi_{2l, 2j-1},\\
E_{-(\varepsilon_{l}-\varepsilon_{j})}&=-\xi_{2l-1, 2j-1}-\xi_{2l, 2j}+i\xi_{2l-1,2j}-i\xi_{2l, 2j-1},\\
E_{-(\varepsilon_{l}+\varepsilon_{j})}&=-\xi_{2l-1, 2j-1}+\xi_{2l, 2j}-i\xi_{2l-1,2j}-i\xi_{2l, 2j-1}.
\end{align*}
We denote by $N_{\mathfrak{so}_{6}}$ the nilpotent subalgebra $\oplus_{\alpha \in \Delta^{+}} \g_{\alpha}$.\\
We introduce the following notation. Given a proposition $P$, we let
\begin{equation*}
\bigchi_{P}=
\begin{cases}
1 \quad \text{if $P$ is true,}\\
0 \quad \text{if $P$ is false.}
\end{cases}
\end{equation*}
From now on $F$ will be a finite$-$dimensional irreducible $\g_{0}-$module, such that $\g_{>0}$ acts trivially on it. We point out that $\Ind(F) \cong \C[\Theta] \otimes \inlinewedge(6)\otimes F$. Indeed, let us denote by $\eta_{i}$ the image in $U(\g)$ of $\xi_{i} \in  \inlinewedge(6) $, for all $i \in \left\{1,2,3,4,5,6\right\}$.
 In $U(\g)$ we have that $\eta_{i}^{2}=\Theta$, for all $i \in \left\{1,2,3,4,5,6\right\}$: since $[\xi_{i},\xi_{i}]=-1$ in $\g$, we have $\eta_{i}\eta_{i}=-\eta_{i}\eta_{i}-1$ in $U(\g)$.
We will make the following abuse of notation: if $I,J\in \mathcal I_{\neq}$ we will denote by $I \cap J$ the increasingly ordered sequence whose elements are the elements of the intersection of the underlying sets of $I$ and $J$. Given $I=i_{1}, \cdots i_{k}\in \mathcal I_{\neq}$, we will use the notation  $\eta_{I}$ to denote the element  $\eta_{i_{1}}  \cdots \eta_{i_{k}} \in  U(\g_{<0})$ and we will denote  $|\eta_{I}|=|I|=k$. We will denote $\eta_{*}=\eta_{123456}$. 
Given $I,J \in \mathcal{I}_{\neq}$, we define:
\begin{align*}
 \xi_{I}  \star \eta_{J}=\bigchi_{I\cap J=\emptyset}\eta_{I}\eta_{J},\\
 \eta_{J}  \star \xi_{I}=\bigchi_{I\cap J=\emptyset}\eta_{J}\eta_{I}.
\end{align*}
We will also use the following notation: if $i_1\cdots i_k \in \mathcal I_{\neq}$ and $i\in \{1,2,3,4,5,6\}$ we let 
\[
\de_i \eta_{i_1,\ldots,i_k}=\begin{cases}
(-1)^{j+1} \eta_{i_1,\ldots,\hat{i_j},\ldots,i_k}& \textrm{if $i=i_j$ for some $j$}\\
0& \textrm{otherwise.}
\end{cases}
\]
and for $a \in \C$, $I=(i_{1} ,\, i_{2}, \cdots i_{k}),J\in \mathcal I_{\neq}$: 
\begin{align*}
\partial_{I}\eta_{J}&=\partial_{i_{1}}\partial_{i_{2}}\dots \partial_{i_{k}}\eta_{J} \ \ &\partial_{I}\xi_{J}&=\partial_{i_{1}}\partial_{i_{2}}\dots \partial_{i_{k}}\xi_{J} ;\\
\partial_{a \xi_{I}}\eta_{J}&=a \partial_{I}\eta_{J} \ \ &\partial_{a \xi_{I}}\xi_{J}&=a \partial_{I}\xi_{J};\\
\partial_{\emptyset}\eta_{S}&=\eta_{S} \ \ &\partial_{\emptyset}\xi_{S}&=\xi_{S}.
\end{align*}
We extend the definition of modified Hodge dual to the elements of $U(\g_{<0})$ in the following way: for $\eta_{I} \in U(\g_{<0})$, we let $\eta_{I}^{*}$ to be the unique monomial such that $\xi_{I}\star  \eta_{I}^{*}=\eta_{*}$. \\
Moreover we define the Hodge dual of elements of $\inlinewedge(6)$ (resp. $U(\g_{<0})$) in the following way: for $\xi_{I} \in \inlinewedge(6)$ (resp. $\eta_{I} \in U(\g_{<0})$), we let $\overline{\xi_{I}}$ (resp. $\overline{\eta_{I}}$) to be the unique monomial such that $\overline{\xi_{I}} \xi_{I} =\xi_{*}$ (resp. $\overline{\eta_{I}}\star \xi_{I} =\eta_{*}$). Then we extend by linearity the definition of Hodge dual to elements $\sum_{I}\alpha_{I}\xi_{I} $ (resp. $\sum_{I}\alpha_{I}\eta_{I} $) and we set $\overline{t^{k}\xi_{I}}=t^{k}\overline{\xi_{I}}$ (resp. $\overline{\Theta^{k}\eta_{I}}=\Theta^{k}\overline{\eta_{I}}$). We point out that for $\eta_{I} \in U(\g_{<0})$, $\overline{\eta_{I}}=(-1)^{|I|}\eta_{I}^{*}$.\\
In order to study singular vectors, it is important to find an explicit form for the action of $\g$ on $\Ind (F)$ using the $\lambda-$action notation \eqref{powerseries}. Due to the fact that the homogeneous components of non$-$positive degree of $E(1, 6)$ are the same as those of $K(1, 6)_{+}$, the $\lambda-$action is given by restricting the $\lambda-$action for $K(1, 6)_{+}$:
\begin{align*}
{\xi_L\,}_{\lambda}(g \otimes v)= \sum_{j \geq 0} \frac{\lambda^{j}}{j!}t^{j}\xi_L .(g \otimes v),
\end{align*}
for $L\in \mathcal I$, $g \otimes v \in \Ind(F)$, described explicitly in Theorem 4.1 in \cite{kac1}.
We recall the following result proved in \cite[Theorem 4.3]{kac1} for the $\lambda-$action in the case of $K(1, 6)_{+}$.
\begin{prop}[\cite{kac1}]
\label{actiondualck6}
Let $T$ be the vector spaces isomorphism $T:\Ind(F)\rightarrow \Ind(F)$, $g \otimes v\mapsto \overline{g}\otimes v$, for all $g \otimes v\in \Ind(F)\cong \C[\Theta] \otimes \inlinewedge(6)\otimes F$. Let $L,I \in \mathcal{I}_{\neq}$. Then
\begin{align*}
T &\circ {\xi_{L}}_{\lambda}\circ T^{-1}(\eta_{I}\otimes v) \\
=&(-1)^{\frac{|L|(|L|+1)}{2}+|L||I|} \left\{ (|L|-2) \Theta (\xi_{L} \star \eta_{I}) \otimes v -(-1)^{|L|} \sum^{6}_{i=1}(\partial_{i}\xi_{L}\star \partial_{i}\eta_{I}) \otimes v -\sum_{r<s}  (\partial_{rs}\xi_{L}\star \eta_{I} )\otimes \xi_{sr}.v  \right.\\
& \left.+\lambda \Big((\xi_{L}\star \eta_{I} ) \otimes t.v-(-1)^{|L|}\sum^{6}_{i=1} \partial_{i}(\xi_{Li} \star \eta_{I} ) \otimes v+ (-1)^{|L|} \sum _{i \neq j} (\partial_{i}\xi_{Lj} \star \eta_{I} )\otimes \xi_{ji}. v\Big) \right. \\
& \left.- \lambda^{2}  \sum _{i < j} (\xi_{Lij} \star \eta_{I}) \otimes   \xi_{ji}. v \right\}.
\end{align*}
\end{prop}
The following lemma allows to compute ${\xi_{L}} _{\lambda}\big(\Theta^{k}\xi_{I} \otimes v\big)$.
 \begin{lem}
\label{lambda+thetack6}
Let $L,I \in \mathcal{I}_{\neq}$ and $k \geq 0$. The following holds:
\begin{align*}
{\xi_{L}}_{\lambda}\big(\Theta^{k}\xi_{I} \otimes v\big)=(\Theta+\lambda)^{k}({\xi_{L}} _{\lambda} \xi_{I}\otimes v).
\end{align*}
\end{lem}
\begin{proof}
The proof is analogous to Lemma 5.11 in \cite{bagnoli}.
\end{proof}
Let $\vec{m}$ be a vector of the $E(1,6)-$module $\Ind(F)$. From \cite{ck6} we know that $\vec{m}$ is a highest weight singular vector if and only if:
\begin{description}
  \item[S0] $N_{\so_{6}}.\vec{m}=0$.
	\item[S1] For all $L \in \mathcal{I}_{\neq}$, with $0 \leq |L| \leq 3$:
	$$\frac{d^{2}}{d \lambda^{2}} \left({\xi_{L}}_{\lambda}\vec{m}-i(-1)^{\frac{|L|(|L|+1)}{2}}\lambda^{3-|L|}\left({\xi_{L}^{*}}_{\lambda}\vec{m}\right)\right)=0.$$
	\item[S2] For all $L \in \mathcal{I}_{\neq}$, with $1 \leq |L| \leq 3$:
	$$\frac{d}{d \lambda} \left({\xi_{L} }_{\lambda}\vec{m}-i(-1)^{\frac{|L|(|L|+1)}{2}}\lambda^{3-|L|}\left({\xi_{L}^{*}}_{\lambda}\vec{m}\right)\right)_{| \lambda=0}=0.$$
	\item[S3]For all $L \in \mathcal{I}_{\neq}$, with $|L| = 3$:
	$$\left({\xi_{L} }_{\lambda}\vec{m}-i(-1)^{\frac{|L|(|L|+1)}{2}}\lambda^{3-|L|}\left({\xi_{L}^{*}}_{\lambda}\vec{m}\right)\right)_{| \lambda=0}=0.$$
\end{description}
In particular condition \textbf{S0} is equivalent to impose that $\vec{m}$ is a highest weight vector; conditions \textbf{S1}--\textbf{S3} are equivalent to impose that $\vec{m}$ is a singular vector. Indeed condition \textbf{S1} is equivalent to
	\begin{align*}
	\sum_{j\geq 2}j(j-1)\frac{\lambda^{j-2}}{j!}(t^{j}\xi_L)\vec{m}-i(-1)^{\frac{|L|(|L|+1)}{2}}\sum_{j\geq 2}(3-|L|+j)(2-|L|+j)\frac{\lambda^{1-|L|+j}}{j!}(t^{j}\xi_L^{*})\vec{m}=0 ,
	\end{align*}
	which implies $(t^{j}\xi_{L}-i(-1)^{\frac{|L|(|L|+1)}{2}}\left(\frac{d}{dt}\right)^{3-|L|}t^{j}\xi_{L}^{*})\vec{m}=0$ for all $L\in \mathcal I_{\neq}$, with $0\leq |L|\leq 3$ and $j\geq 2$.\\
	Condition \textbf{S2} is equivalent to $(t\xi_{L}-i(-1)^{\frac{|L|(|L|+1)}{2}}\left(\frac{d}{dt}\right)^{3-|L|}t\xi_{L}^{*})\vec{m}=0 $ for all for all $L\in \mathcal I_{\neq}$, with $ 1 \leq |L| \leq 3$.\\
	Condition \textbf{S3} is equivalent to $(\xi_{L}-i\xi_{L}^{*}) \vec{m}=0 $ for all $L\in \mathcal I_{\neq}$ such that $ |L| =3$.
	Therefore, by Remark \ref{costruzioneannihi}, \textbf{S1}--\textbf{S3} are equivalent to impose that $\vec{m}$ is a singular vector.
\begin{rem}	
\label{appoggiotecnicos1s2s3ck6}
We point out that, by the previous conditions, a vector $\vec{m} \in \Ind (F)$ is a highest weight singular vector if and only if it satisfies  \textbf{S0}--\textbf{S3}. Since $T$, defined as in Proposition \ref{actiondualck6}, is an isomorphism, the fact that $\vec{m} \in \Ind (F)$ satisfies \textbf{S0}--\textbf{S3} is equivalent to impose \textbf{S0}--\textbf{S3} for $(T \circ ({\xi_{L}} _{\lambda}-i(-1)^{\frac{|L|(|L|+1)}{2}}\lambda^{3-|L|}{\xi_{L}^{*}}_{\lambda}) \circ T^{-1})T(\vec{m})$, using the expression given by Proposition \ref{actiondualck6}.\\
Therefore in the following results we will consider a vector $T(\vec{m}) \in \Ind (F)$ and we will impose that the expression for $(T \circ ({\xi_{L}} _{\lambda}-i(-1)^{\frac{|L|(|L|+1)}{2}}\lambda^{3-|L|}{\xi_{L}^{*}}_{\lambda}) \circ T^{-1})T(\vec{m})=(T \circ ({\xi_{L}} _{\lambda}-i(-1)^{\frac{|L|(|L|+1)}{2}}\lambda^{3-|L|}{\xi_{L}^{*}}_{\lambda}) ) \vec{m}$ given by Proposition \ref{actiondualck6} satisfies conditions \textbf{S0}--\textbf{S3}. We will have that $\vec{m}$ is a highest weight singular vector.
\end{rem}	
Motivated by Remark \ref{appoggiotecnicos1s2s3ck6}, we consider a singular vector $\vec{m} \in \Ind (F)$ such that:
\begin{align}
\label{singck6}
T(\vec{m}) =\sum_{k=0} ^{N}\Theta^{k} \sum_{I \in \mathcal{I}_{<}}\eta_{I} \otimes v_{I,k}.
\end{align}
We will denote $ v_{123456,k}= v_{*,k}$ for all $k$.
\section{Main result}
In \cite{ck6} the following Lemma is stated without proof (Lemma $4.4$ in \cite{ck6}). In particular, this Lemma is used in \cite{ck6} to completely classify the highest weight singular vectors of finite Verma modules over $\g$.
\begin{lem}
\label{lemck6grado}
Let $\vec{m}\in \Ind(F)$ be a singular vector, such that $T(\vec{m})$ is written as in \eqref{singck6}. Then the degree of $\vec{m}$ with respect to $\Theta$ is at most 2.
Moreover, $T(\vec{m})$ has the following form:
\begin{align*}
T(\vec{m}) = \Theta^{2} \sum_{|I|\geq 5}\eta_{I} \otimes v_{I,2}+\Theta \sum_{|I|\geq 3}\eta_{I} \otimes v_{I,1}+ \sum_{|I|\geq 1}\eta_{I} \otimes v_{I,0}.
\end{align*}
\end{lem}
The rest of this section is the dedicated to the proof of Lemma \ref{lemck6grado}.
\begin{lem}
\label{lemmagrado4}
A singular vector $\vec{m}\in \Ind(F)$, such that $T(\vec{m})$ is written as in \eqref{singck6}, has degree at most 4 with respect to $\Theta$.
\end{lem}

\begin{proof}
By Remark \ref{appoggiotecnicos1s2s3ck6}, condition \textbf{S1} for $\xi_{1}$ reduces to:
\begin{align*}
\frac{d^{2}}{d \lambda^{2}} \left(T({\xi_{1} \,}_{\lambda}\vec{m}+i\lambda^{2}({\xi_{23456}\,}_{\lambda}\vec{m}))\right)=0.
\end{align*}
Using Proposition \ref{actiondualck6} and Lemma \ref{lambda+thetack6}, the previous equation reduces to:
\begin{align}
\label{appoggiolemmatecnico}
&0=\frac{d^{2}}{d \lambda^{2}} \sum_{k=0}^{N}\sum_{I}(\lambda+\Theta)^{k}(-1)^{1+|I|}\Bigg\{ \Bigg[-\Theta (\xi_{1}\star \eta_{I}) \otimes v_{I,k}+\partial_{1}\eta_{I} \otimes v_{I,k}+\lambda \bigg((\xi_{1}\star \eta_{I})\otimes t. v_{I,k}\\ \nonumber
&+\sum_{l=1}^{6}\partial_{l}(\xi_{1l}\star \eta_{I})\otimes v_{I,k}-\sum_{j \neq 1} (\xi_{j}\star \eta_{I}) \otimes \xi_{j1}.v_{I,k}\bigg)-\lambda^{2}\sum_{l<j} (\xi_{1lj} \star \eta_{I}) \otimes \xi_{jl}. v_{I,k}\Bigg]\\ \nonumber
&+i\lambda^{2}\Bigg[3\Theta (\xi_{23456}\star \eta_{I})\otimes v_{I,k}+\sum_{l=1}^{6}(\partial_{l}\xi_{23456}\star \partial_{l}\eta_{I}) \otimes v_{I,k}-\sum_{r<s} (\partial_{rs}\xi_{23456}\star \eta_{I}) \otimes \xi_{sr}.v_{I,k}\\ \nonumber
&+\lambda\bigg((\xi_{23456}\star \eta_{I} )\otimes t.v_{I,k}+\sum_{l=1}^{6}\partial_{l}(\xi_{23456l}\star \eta_{I})\otimes v_{I,k}-\sum_{l \neq j} (\partial_{l}\xi_{23456j}\star \eta_{I})\otimes \xi_{jl}.v_{I,k}\bigg)\Bigg] \Bigg\}\\ \nonumber
&=\sum_{k=0}^{N} \sum_{I} (\lambda+\Theta)^{k} (-1)^{1+|I|}\bigg(-2 \sum_{l<j} (\xi_{1lj}\star \eta_{I}) \otimes \xi_{jl} . v_{I,k}\bigg)\\ \nonumber
&+2 \sum _{k=1}^{N} \sum _{I} k  (\lambda+\Theta)^{k-1} (-1)^{1+|I|} \bigg[(\xi_{1}\star \eta_{I}) \otimes t. v_{I,k}+\sum_{l=1}^{6}\partial_{l}(\xi_{1l}\star \eta_{I}  )\otimes v_{I,k}-\sum_{j \neq 1} (\xi_{j}\star \eta_{I}) \otimes \xi_{j1}. v_{I,k}\\ \nonumber
&-2 \lambda  \sum_{l<j} (\xi_{1lj} \star \eta_{I}) \otimes \xi_{jl} .v_{I,k} \bigg]+ \sum _{k=2}^{N} \sum _{I} k(k-1) (\lambda+\Theta)^{k-2} (-1)^{1+|I|} \bigg[-\Theta (\xi_{1} \star \eta_{I}) \otimes v_{I,k}+\partial_{1}\eta_{I} \otimes  v_{I,k}\\ \nonumber
&+\lambda\bigg((\xi_{1} \star \eta_{I}) \otimes t.v_{I,k}+\sum^{6}_{l=1} \partial_{l}(\xi_{1l}\star \eta_{I} )\otimes v_{I,k}-\sum_{j \neq 1} (\xi_{j}\star \eta_{I} )\otimes \xi_{j1}.  v_{I,k}\bigg)-\lambda^{2}\sum_{l<j}( \xi_{1lj} \star \eta_{I} )\otimes \xi_{jl}.  v_{I,k}\bigg]\\ \nonumber
&+\sum_{k=0}^{N} \sum_{I}(-1)^{1+|I|}\big(2i(\lambda+\Theta)^{k}+4i\lambda k(\lambda+\Theta)^{k-1} +i\lambda^{2} k(k-1)(\lambda+\Theta)^{k-2} \big)\cdot \\ \nonumber
&\bigg[3 \Theta (\xi_{23456}\star \eta_{I}) \otimes  v_{I,k}+\sum_{l=1}^{6}( \partial_{l}\xi_{23456}\star \partial_{l}\eta_{I})\otimes  v_{I,k}-\sum_{r<s} (\partial_{rs}\xi_{23456} \star \eta_{I}) \otimes \xi_{sr}.  v_{I,k}\\ \nonumber
&+\lambda \bigg((\xi_{23456}\star \eta_{I} )\otimes t.v_{I,k}+\sum^{6}_{l=1} \partial_{l}(\xi_{23456l}\star \eta_{I}) \otimes  v_{I,k}- \sum_{l \neq j} (\partial_{l}\xi_{23456j} \star \eta_{I}) \otimes \xi_{jl}.  v_{I,k} \bigg)\bigg]\\ \nonumber
&+\sum^{N}_{k=0} \sum_{I} (-1)^{1+|I|}\big(4i\lambda (\lambda+\Theta)^{k}+2 i\lambda^{2}k (\lambda+\Theta)^{k-1}  \big)\bigg[(\xi_{23456}\star \eta_{I} )\otimes t. v_{I,k}\\ \nonumber
&+\sum_{l=1}^{6}\partial_{l}(\xi_{23456l}\star \eta_{I}) \otimes  v_{I,k}-\sum_{l \neq j} (\partial_{l} \xi_{23456j}\star \eta_{I} )\otimes \xi_{jl}.  v_{I,k}\bigg].
\end{align}
We consider the previous expression as a polynomial in $\lambda$ and $\lambda+\Theta$, by writing $\Theta$ as $(\lambda+\Theta)-\lambda$.\\
We look at the coefficient of $\lambda^{3}(\lambda+\Theta)^{s}$, for a fixed $s\geq 0$, in \eqref{appoggiolemmatecnico} and we obtain that:
\begin{align}
\label{alfa}
 &\sum_{I} (-1)^{1+|I|}   \bigg[-3 ( \xi_{23456}\star \eta_{I}) \otimes  v_{I,s+2}+(\xi_{23456}\star \eta_{I}) \otimes t .v_{I,s+2}\\ \nonumber
&+\sum_{l=1}^{6}\partial_{l}(\xi_{23456l}\star \eta_{I}) \otimes  v_{I,s+2}-\sum_{l \neq j} (\partial_{l} \xi_{23456j} \star \eta_{I} )\otimes \xi_{jl}. v_{I,s+2}\bigg]=0.
\end{align}
We consider the coefficient of $\lambda^{2}(\lambda+\Theta)^{s}$, for a fixed $s\geq 1$, in \eqref{appoggiolemmatecnico} and we obtain that:
\begin{align*}
& \sum _{I}(-1)^{1+|I|}(s+1)\bigg\{ -(s+2)   \sum_{l<j} (\xi_{1lj}\star \eta_{I}) \otimes \xi_{jl}. v_{I,s+2}\\
&+4i \bigg[-3 ( \xi_{23456}\star \eta_{I}) \otimes  v_{I,s+1}+(\xi_{23456}\star \eta_{I} )\otimes t. v_{I,s+1}\\ \nonumber
&+\sum_{l=1}^{6}\partial_{l}(\xi_{23456l}\star \eta_{I}) \otimes  v_{I,s+1}-\sum_{l \neq j}( \partial_{l} \xi_{23456j}\star \eta_{I}) \otimes \xi_{jl}.  v_{I,s+1}\bigg]\\
&+i s 3  (\xi_{23456}\star \eta_{I} )\otimes  v_{I,s+1}+ i (s+2) \bigg[ \sum_{l=1}^{6}(\partial_{l}\xi_{23456}\star \partial_{l} \eta_{I}) \otimes  v_{I,s+2}-\sum_{r<p}( \partial_{rp}\xi_{23456}\star \eta_{I}) \otimes \xi_{pr}. v_{I,s+2}\bigg]\\
&+ 2 i\bigg[(\xi_{23456}\star \eta_{I}) \otimes t. v_{I,s+1}+\sum_{l=1}^{6}\partial_{l}(\xi_{23456l}\star \eta_{I}) \otimes  v_{I,s+1}-\sum_{l \neq j}( \partial_{l} \xi_{23456j}\star \eta_{I}) \otimes \xi_{jl}.  v_{I,s+1}\bigg]\bigg\}=0.
\end{align*}
Using \eqref{alfa}, we obtain that the sum over $I$ of the terms in the second and third rows is zero, and the sum over $I$ of the last row is equal to $\sum_{I} (-1)^{1+|I|} 6i ( \xi_{23456}\star \eta_{I} )\otimes  v_{I,s+1}$. 
Hence for $s \geq 1$:
\begin{align}
\label{beta}
& \sum _{I} (-1)^{1+|I|}\bigg\{ -\sum_{l<j} (\xi_{1lj}\star \eta_{I}) \otimes \xi_{jl}.  v_{I,s+2}+3i (\xi_{23456}\star \eta_{I} )\otimes  v_{I,s+1}\\  \nonumber
&+ i  \bigg[ \sum_{l=1}^{6}(\partial_{l}\xi_{23456}\star \partial_{l} \eta_{I}) \otimes  v_{I,s+2}-\sum_{r<p} (\partial_{rp}\xi_{23456}\star \eta_{I} )\otimes \xi_{pr}. v_{I,s+2}\bigg] \bigg\} =0.  
\end{align}
We consider the coefficient of $\lambda(\lambda+\Theta)^{s}$, for a fixed $s\geq 2$, in \eqref{appoggiolemmatecnico} and we obtain that:
\begin{align*}
&\sum _{I}(-1)^{1+|I|}\bigg\{  -4  (s+1) \sum_{l<j} (\xi_{1lj}\star \eta_{I} )\otimes \xi_{jl}. v_{I,s+1} \\
&+  (s+1)(s+2)   \bigg[( \xi_{1}\star \eta_{I}) \otimes v_{I,s+2}+(\xi_{1}\star \eta_{I}) \otimes t.v_{I,s+2}+\sum^{6}_{l=1} \partial_{l}(\xi_{1l}\star \eta_{I} )\otimes v_{I,s+2}\\
&-\sum_{j \neq 1} (\xi_{j}\star \eta_{I}) \otimes \xi_{j1}.  v_{I,s+2}\bigg]\\ \nonumber
&+ 2i \bigg[-3  (\xi_{23456}\star \eta_{I}) \otimes  v_{I,s}+(\xi_{23456}\star \eta_{I}) \otimes t.v_{I,s}\\ \nonumber
&+\sum^{6}_{l=1} \partial_{l}(\xi_{23456l}\star \eta_{I}) \otimes  v_{I,s}- \sum_{l \neq j} (\partial_{l}\xi_{23456j}\star \eta_{I} )\otimes \xi_{jl}.  v_{I,s} \bigg]\\
&+ 12is  (\xi_{23456}\star \eta_{I}) \otimes  v_{I,s}+  4i (s+1) \bigg[ \sum_{l=1}^{6}(\partial_{l}\xi_{23456}\star \partial_{l} \eta_{I}) \otimes  v_{I,s+1}-\sum_{r<p} \partial_{rp}(\xi_{23456}\star \eta_{I}) \otimes \xi_{pr}. v_{I,s+1}\bigg]\\
&+ 4i\bigg[(\xi_{23456}\star \eta_{I} )\otimes t. v_{I,s}+\sum_{l=1}^{6}\partial_{l}(\xi_{23456l}\star \eta_{I}) \otimes  v_{I,s}-\sum_{l \neq j} (\partial_{l} \xi_{23456j}\star \eta_{I}) \otimes \xi_{jl}.  v_{I,s}\bigg]\bigg\}=0.
\end{align*}
We use \eqref{alfa} to point out that the sum over $I$ of the terms in the fourth and fifth rows is zero. Moreover, due to \eqref{alfa}, the sum over $I$ of the terms in the last row is equal to $\sum_{I}(-1)^{1+|I|}12i ( \xi_{23456}\star \eta_{I}) \otimes  v_{I,s}$.
Finally the sum of $\sum_{I}(-1)^{1+|I|}12i  (\xi_{23456}\star \eta_{I} )\otimes  v_{I,s}$ plus the sum over $I$ of the terms from the first and sixth rows is zero due to \eqref{beta}.\\
Therefore for $s \geq 2$:
\begin{align}
\label{gamma}
&\sum _{I}   (-1)^{1+|I|} \bigg[ (\xi_{1}\star \eta_{I}) \otimes v_{I,s+2}+(\xi_{1}\star \eta_{I}) \otimes t.v_{I,s+2}+\sum^{6}_{l=1} \partial_{l}(\xi_{1l}\star \eta_{I} )\otimes v_{I,s+2}\\ \nonumber
&-\sum_{j \neq 1} (\xi_{j}\star \eta_{I}) \otimes \xi_{j1}. v_{I,s+2}\bigg]=0.
\end{align}
Finally we consider the coefficient of $(\lambda+\Theta)^{s}$, for a fixed $s\geq 3$, in \eqref{appoggiolemmatecnico} and we obtain that:
\begin{align*}
&\sum_{I} (-1)^{1+|I|}\bigg\{-2 \sum_{l<j} (\xi_{1lj}\star \eta_{I}) \otimes \xi_{jl} . v_{I,s}\\
&+2  (s+1)   \bigg[(\xi_{1}\star \eta_{I}) \otimes t. v_{I,s+1}+\sum_{l=1}^{6}\partial_{l}(\xi_{1l}\star \eta_{I} ) \otimes v_{I,s+1}-\sum_{j \neq 1} (\xi_{j}\star \eta_{I}) \otimes \xi_{j1}. v_{I,s+1} \bigg]\\
&-   s(s+1)   (\xi_{1}\star \eta_{I}) \otimes v_{I,s+1}+  (s+1)(s+2)   \partial_{1}\eta_{I} \otimes  v_{I,s+2}\\
&+ 6i (\xi_{23456}\star \eta_{I} )\otimes  v_{I,s-1}+ 2i \bigg[\sum_{l=1}^{6} (\partial_{l}\xi_{23456}\star \partial_{l}\eta_{I})\otimes  v_{I,s}-\sum_{r<p}( \partial_{rp}\xi_{23456} \star \eta_{I}) \otimes \xi_{pr}.  v_{I,s}\bigg]\bigg\}=0.
\end{align*}
Using \eqref{beta}, we observe that the sum over $I$ of the terms from the first and the last row is zero. Using \eqref{gamma} we obtain that the sum of the terms from the second row is equal to $-2 \sum _{I} (s+1)  (-1)^{1+|I|} (\xi_{1}\star \eta_{I}) \otimes  v_{I,s+1}$. 
Thus for $s \geq 3$:
\begin{align}
 \label{delta}
&\sum _{I}   (-1)^{1+|I|} (  (\xi_{1}\star \eta_{I}) \otimes v_{I,s+1}-  \partial_{1}\eta_{I} \otimes  v_{I,s+2})=0.
\end{align}
By linear independence, we obtain:
\begin{align*}
&\sum _{I}   (-1)^{1+|I|}  (\xi_{1}\star \eta_{I}) \otimes v_{I,s+1}=0. 
\end{align*}
Therefore $v_{I,k}=0$ for $|I|\leq 5$, $1 \notin I$ and $k \geq 4$. We point out that $1 \notin I$ is not necessary, since we could have chosen at the beginning any $\xi_{i}$ instead of $\xi_{1}$.
Finally, the coefficient of $\eta_{1}^{*}$ in \eqref{delta} is $v_{*,s+2}$.
Hence $v_{*,k}=0$ if $k\geq 5$.
\end{proof}
By Lemma \ref{lemmagrado4}, for a singular vector $\vec{m}$, $T(\vec{m})$ has the following form:
\begin{align}
\label{ck6max4}
T(\vec{m} )= \Theta^{4} \sum_{I\in \mathcal{I}_{<} }\eta_{I} \otimes v_{I,4}+  \Theta^{3} \sum_{I \in \mathcal{I}_{<}}\eta_{I} \otimes v_{I,3}+\Theta^{2} \sum_{I\in \mathcal{I}_{<}}\eta_{I} \otimes v_{I,2}+\Theta \sum_{I\in \mathcal{I}_{<}}\eta_{I} \otimes v_{I,1}+ \sum_{I\in \mathcal{I}_{<}}\eta_{I} \otimes v_{I,0}.
\end{align}
Following \cite{ck6}, we write the $\lambda-$action in the following way, using Proposition \ref{actiondualck6} and Lemma \ref{lambda+thetack6}:
\begin{align*}
T ( {\xi_{L}}_{\lambda}\vec{m})=&b_{0}(\xi_{L})+\lambda(B_{0}(\xi_{L})-a_{0}(\xi_{L}))+\lambda^{2}C_{0}(\xi_{L})\\ \nonumber
&+(\lambda+\Theta)[a_{0}(\xi_{L})+b_{1}(\xi_{L})]+ (\lambda+\Theta)\lambda(B_{1}(\xi_{L})-a_{1}(\xi_{L}))+ (\lambda+\Theta)\lambda^{2}C_{1}(\xi_{L})\\ \nonumber
&+(\lambda+\Theta)^{2}[a_{1}(\xi_{L})+b_{2}(\xi_{L})]+(\lambda+\Theta)^{2}\lambda(B_{2}(\xi_{L})-a_{2}(\xi_{L}))+(\lambda+\Theta)^{2}\lambda^{2}C_{2}(\xi_{L})\\
&+(\lambda+\Theta)^{3}[a_{2}(\xi_{L})+b_{3}(\xi_{L})]+(\lambda+\Theta)^{3}\lambda(B_{3}(\xi_{L})-a_{3}(\xi_{L}))+(\lambda+\Theta)^{3}\lambda^{2}C_{3}(\xi_{L})\\
&+(\lambda+\Theta)^{4}[a_{3}(\xi_{L})+b_{4}(\xi_{L})]+(\lambda+\Theta)^{4}\lambda(B_{4}(\xi_{L})-a_{4}(\xi_{L}))+(\lambda+\Theta)^{4}\lambda^{2}C_{4}(\xi_{L})	\\ \nonumber
&+(\lambda+\Theta)^{5}a_{4}(\xi_{L}),
\end{align*}
where the coefficients $a_{p}(\xi_{L}),b_{p}(\xi_{L}),B_{p}(\xi_{L}),C_{p}(\xi_{L})$ depend on $\xi_{L}$ for all $0\leq p \leq 4$ and are explicitly defined as follows. For all $0\leq p\leq 4$ we let:
\begin{align}
\label{abcdck6}
a_{p}(\xi_{L})&=\sum_{I} (-1)^{(|L|(|L|+1)/2)+|L||I|} \bigg[ (|L|-2)  (\xi_{L} \star \eta_{I}) \otimes v_{I,p}\bigg];\\ \nonumber
b_{p}(\xi_{L})&=\sum_{I} (-1)^{(|L|(|L|+1)/2)+|L||I|} \bigg[ -(-1)^{|L|} \sum^{6}_{i=1}(\partial_{i}\xi_{L}\star \partial_{i}\eta_{I}) \otimes v_{I,p} -\sum_{r<s}  (\partial_{rs}\xi_{L}\star \eta_{I})\otimes \xi_{sr}.v_{I,p} \bigg];\\ \nonumber
B_{p}(\xi_{L})&=\sum_{I} (-1)^{(|L|(|L|+1)/2)+|L||I|} \bigg[(\xi_{L} \star \eta_{I} ) \otimes t.v_{I,p}-(-1)^{|L|}\sum^{6}_{i=1} \partial_{i}(\xi_{Li} \star \eta_{I} ) \otimes v_{I,p}\\ \nonumber
&+ (-1)^{|L|} \sum _{i \neq j} (\partial_{i}\xi_{Lj} \star \eta_{I}) \otimes \xi_{ji}. v_{I,p} \bigg];\\ \nonumber
C_{p}(\xi_{L})&=\sum_{I} (-1)^{(|L|(|L|+1)/2)+|L||I|} \bigg[- \sum _{i < j} (\xi_{Lij}\star \eta_{I})\otimes   \xi_{ji}.v_{I,p}\bigg].
\end{align}
We will write $a_{p}$ instead of $a_{p}(\xi_{L})$ if there is no risk of confusion, and similarly for the others.
Analogously:
\begin{align}
\label{abcdck6dual}
T({ \xi_{L}^{*}}_{\lambda}\vec{m})=&bd_{0}(\xi_{L})+\lambda(Bd_{0}(\xi_{L})-ad_{0}(\xi_{L}))+\lambda^{2}Cd_{0}(\xi_{L})\\ \nonumber
&+(\lambda+\Theta)[ad_{0}(\xi_{L})+bd_{1}(\xi_{L})]+ (\lambda+\Theta)\lambda(Bd_{1}(\xi_{L})-ad_{1}(\xi_{L}))+ (\lambda+\Theta)\lambda^{2}Cd_{1}(\xi_{L})\\ \nonumber
&+(\lambda+\Theta)^{2}[ad_{1}(\xi_{L})+bd_{2}(\xi_{L})]+(\lambda+\Theta)^{2}\lambda(Bd_{2}(\xi_{L})-ad_{2}(\xi_{L}))+(\lambda+\Theta)^{2}\lambda^{2}Cd_{2}(\xi_{L})\\ \nonumber
&+(\lambda+\Theta)^{3}[ad_{2}(\xi_{L})+bd_{3}(\xi_{L})]+(\lambda+\Theta)^{3}\lambda(Bd_{3}(\xi_{L})-ad_{3}(\xi_{L}))+(\lambda+\Theta)^{3}\lambda^{2}Cd_{3}(\xi_{L})\\ \nonumber
&+(\lambda+\Theta)^{4}[ad_{3}(\xi_{L})+bd_{4}(\xi_{L})]+(\lambda+\Theta)^{4}\lambda(Bd_{4}(\xi_{L})-ad_{4}(\xi_{L}))+(\lambda+\Theta)^{4}\lambda^{2}Cd_{4}(\xi_{L})\\ \nonumber
&+(\lambda+\Theta)^{5}ad_{4}(\xi_{L}),
%
%
%
%
\end{align}
where $ad_{p}(\xi_{L})=a_{p}(\xi_{L}^{*})$, $bd_{p}(\xi_{L})=b_{p}(\xi_{L}^{*})$, $Bd_{p}(\xi_{L})=B_{p}(\xi_{L}^{*})$, $Cd_{p}(\xi_{L})=C_{p}(\xi_{L}^{*})$.
We will write $ad_{p}$ instead of $ad_{p}(\xi_{L})$ if there is no risk of confusion, and similarly for the others. We will shortly write $a_{p}(L)$ (resp.  $ad_{p}(L)$) instead of $ad_{p}(\xi_{L})$ (resp. $ad_{p}(\xi_{L})$) when we need to explicit the dependence and similarly for all the others.
\begin{lem}
\label{lemmatecnicoperck6grado}
Let $\vec{m}$ be a singular vector, such that $T(\vec{m})$ is written as in \eqref{ck6max4}.
\begin{enumerate}
	\item[(i)] Condition \textbf{S2} for $L=j$ implies:
	\begin{align*}
4a_{4}+B_{4}=3a_{3}+B_{3}+4b_{4}=2a_{2}+B_{2}+3b_{3}=B_{1}+a_{1}+2b_{2}=B_{0}+b_{1}=0.
\end{align*}
\item[(ii)] Condition \textbf{S2} for $L=ijk$ implies:
\begin{align*}
&4a_{4}+B_{4}-i(4ad_{4}+Bd_{4})=3a_{3}+B_{3}+4b_{4}-i(3ad_{3}+Bd_{3}+4bd_{4})\\
&=2a_{2}+B_{2}+3b_{3}-i(2ad_{2}+Bd_{2}+3bd_{3})=B_{1}+a_{1}+2b_{2}-i(Bd_{1}+ad_{1}+2bd_{2})\\
&= B_{0}+b_{1}-i(Bd_{0}+bd_{1})=0.
\end{align*}
\item[(iii)] Condition \textbf{S3} for $L=ijk$ implies:
\begin{align*}
&a_{4}-iad_{4}=a_{3}+b_{4}-i(ad_{3}+bd_{4})=a_{2}+b_{3}-i(ad_{2}+bd_{3})=a_{1}+b_{2}-i(ad_{1}+bd_{2})\\
&=a_{0}+b_{1}-i(ad_{0}+bd_{1})=b_{0}-ibd_{0}=0.
\end{align*}
\item[(iv)] Condition \textbf{S1} for $|L|=0$ implies:
\begin{align}
&C_{3}+4B_{4}+6a_{4}=0, \label{11ck6}\\
&  C_{2}+3a_{3}+3B_{3}+6b_{4}=0, \label{12ck6}\\
&2C_{3}+2(B_{4}-a_{4})-iad_{1}-ibd_{2}=0, \label{21ck6}\\
&4C_{2}+3B_{3}-3a_{3}-3iad_{0}-3ibd_{1}=0, \label{26ck6}\\
 &C_{3}  -2iBd_{1}-2ibd_{2}=0,  \label{27ck6}\\
 &C_{2}   -6iBd_{0}  +3i ad_{0} -3ibd_{1}=0, \label{31ck6}  \\                           
&    10Cd_{0}    +4Bd_{1} -3ad_{1}+bd_{2}=0. \label{32ck6}   
\end{align}
\end{enumerate}
\end{lem}
\begin{proof}
It follows by direct computations using notation \eqref{abcdck6} and \eqref{abcdck6dual}.
\end{proof}
\begin{lem}
\label{lemmatecnicoperck6grado2}
Let $\vec{m}$ be a singular vector, such that $T(\vec{m})$ is written as in \eqref{ck6max4}.
\item Condition \textbf{S2} for $L=j$ implies:
	\begin{align}
\label{B1+a1+2b2}
0=&\sum_{I} (-1)^{1+|I|} \bigg[(\xi_{j} \star \eta_{I} ) \otimes t.v_{I,1}+\sum^{6}_{i=1} \partial_{i}(\xi_{ji} \star \eta_{I} ) \otimes v_{I,1}- \sum _{i \neq l} (\partial_{i}\xi_{jl} \star \eta_{I}) \otimes \xi_{li}. v_{I,1}\\ \nonumber
& -  (\xi_{j} \star \eta_{I}) \otimes v_{I,1}+2   \partial_{j}\eta_{I} \otimes v_{I,2} \bigg];
\end{align}
\begin{align}
\label{B0+b1}
&0=\sum_{I} (-1)^{1+|I|} \bigg[(\xi_{j} \star \eta_{I} ) \otimes t.v_{I,0}+\sum^{6}_{i=1} \partial_{i}(\xi_{ji} \star \eta_{I} ) \otimes v_{I,0}- \sum _{i \neq l} (\partial_{i}\xi_{jl} \star \eta_{I}) \otimes \xi_{li}. v_{I,0}+ \partial_{j}\eta_{I} \otimes v_{I,1}\bigg];
\end{align}
\begin{align}
\label{2a2+B2+3b3}
&0=\sum_{I} (-1)^{1+|I|} \bigg[ -2  (\xi_{j} \star \eta_{I}) \otimes v_{I,2} +(\xi_{j} \star \eta_{I} ) \otimes t.v_{I,2}\\ \nonumber
&+\sum^{6}_{i=1} \partial_{i}(\xi_{ji} \star \eta_{I} ) \otimes v_{I,2}- \sum _{i \neq l} (\partial_{i}\xi_{jl} \star \eta_{I}) \otimes \xi_{li}. v_{I,2}
+3    \partial_{j}\eta_{I} \otimes v_{I,3} \bigg].
\end{align}
\item Conditions \textbf{S2} and \textbf{S3} for $L=ijk$ imply:
\begin{align}
\label{B1-a1}
&0=\sum_{I} (-1)^{|I|} \bigg[(\xi_{ijk} \star \eta_{I} ) \otimes t.v_{I,1}+\sum^{6}_{l=1} \partial_{l}(\xi_{ijkl} \star \eta_{I} ) \otimes v_{I,1}- \sum _{h \neq l} (\partial_{h}\xi_{ijkl} \star \eta_{I}) \otimes \xi_{lh}. v_{I,1}\\\nonumber
&-   (\xi_{ijk} \star \eta_{I}) \otimes v_{I,1}-i \bigg( (\xi^{*}_{ijk} \star \eta_{I} ) \otimes t.v_{I,1}+\sum^{6}_{l=1} \partial_{l}(\xi^{*}_{ijk}\xi_{l} \star \eta_{I} ) \otimes v_{I,1} \\ \nonumber
&- \sum _{h \neq l} (\partial_{h}\xi^{*}_{ijk}\xi_{l} \star \eta_{I}) \otimes \xi_{lh}. v_{I,1} -   (\xi_{ijk}^{*} \star \eta_{I}) \otimes v_{I,1}  \bigg)\bigg];
\end{align}
\begin{align}
\label{a0-B0}
&0=\sum_{I} (-1)^{|I|} \bigg[   (\xi_{ijk} \star \eta_{I}) \otimes v_{I,0}-(\xi_{ijk} \star \eta_{I} ) \otimes t.v_{I,0}-\sum^{6}_{l=1} \partial_{l}(\xi_{ijkl} \star \eta_{I} ) \otimes v_{I,0}\\ \nonumber
&+ \sum _{h \neq j} (\partial_{h}\xi_{ijkl} \star \eta_{I}) \otimes \xi_{lh}. v_{I,0}-i \bigg( (\xi^{*}_{ijk} \star \eta_{I}) \otimes v_{I,0}-(\xi^{*}_{ijk} \star \eta_{I} ) \otimes t.v_{I,0}\\\nonumber
&-\sum^{6}_{l=1} \partial_{l}(\xi^{*}_{ijk} \xi_{l} \star \eta_{I} ) \otimes v_{I,0}+ \sum _{h \neq j} (\partial_{h}\xi^{*}_{ijk} \xi_{l} \star \eta_{I}) \otimes \xi_{lh}. v_{I,0}  \bigg)\bigg];
\end{align}
\begin{align}
\label{-a2+B2}
&0=\sum_{I} (-1)^{|I|} \bigg[ -  (\xi_{ijk} \star \eta_{I}) \otimes v_{I,2}   +  (\xi_{ijk} \star \eta_{I} ) \otimes t.v_{I,2}+\sum^{6}_{l=1} \partial_{l}(\xi_{ijkl} \star \eta_{I} ) \otimes v_{I,2}\\ \nonumber
&- \sum _{h \neq l} (\partial_{h}\xi_{ijkl} \star \eta_{I}) \otimes \xi_{lh}. v_{I,2} -i \bigg( -(\xi^{*}_{ijk} \star \eta_{I}) \otimes v_{I,2}   +  (\xi^{*}_{ijk} \star \eta_{I} ) \otimes t.v_{I,2}\\ \nonumber
&+\sum^{6}_{l=1} \partial_{l}(\xi^{*}_{ijk}\xi_{l} \star \eta_{I} ) \otimes v_{I,2}- \sum _{h \neq l} (\partial_{h}\xi^{*}_{ijk}\xi_{l} \star \eta_{I}) \otimes \xi_{lh}. v_{I,2} \bigg)\bigg].
\end{align}
\begin{proof}
These are the explicit expression of some of equations of Lemma \ref{lemmatecnicoperck6grado}.
Equation \eqref{B1+a1+2b2} is $B_1(j)+a_1(j)+2b_2(j)=0$, \eqref{B0+b1} is $B_{0}(j)+b_{1}(j)=0$, \eqref{2a2+B2+3b3} is $2a_{2}(j)+B_{2}(j)+3b_{3}(j)=0$.
By Lemma \ref{lemmatecnicoperck6grado}, relations \textbf{S2} and \textbf{S3} for $L=ijk$ imply, taking linear combinations:
\begin{align*}
B_{1}-a_{1}-i(Bd_{1}-ad_{1})=a_{0}-B_{0}-i(ad_{0}-Bd_{0})=-a_{2}+B_{2}-i(-ad_{2}+Bd_{2})=0.
\end{align*}
Equation \eqref{B1-a1} is $B_{1}(ijk)-a_{1}(ijk)-i(Bd_{1}(ijk)-ad_{1}(ijk))=0$, equation \eqref{a0-B0} is $a_{0}(ijk)-B_{0}(ijk)-i(ad_{0}(ijk)-Bd_{0}(ijk))=0$, equation \eqref{-a2+B2} is $-a_{2}(ijk)+B_{2}(ijk)-i(-ad_{2}(ijk)+Bd_{2}(ijk))=0$.
\end{proof}
\end{lem}
\begin{proof}[Proof of Lemma \ref{lemck6grado}]
By Lemma \ref{lemmagrado4}, for a singular vector $\vec{m}$, $T(\vec{m})$ is written as in \eqref{ck6max4}.
Let us consider \eqref{B1+a1+2b2} for $L=j$; the coefficient of $\eta_{j}$ is:
\begin{align} 
\label{tecres} 
t.v_{\emptyset,1}-6v_{\emptyset,1}=0.
\end{align}
Let us consider \eqref{B0+b1} for $L=j$; the coefficient of $\eta_{j}$ is:
\begin{align}
\label{tecres2} 
t.v_{\emptyset,0}-5v_{\emptyset,0}=0.
\end{align}
Let us consider \eqref{2a2+B2+3b3} for $L=j$; the coefficient of $\eta_{j}$ is:
\begin{align}
\label{tecres3} 
t.v_{\emptyset,2}-7v_{\emptyset,2}=0.
\end{align}
The coefficient of 1 in \eqref{B0+b1} for $L=j$ is $v_{j,1}=0$.
The coefficient of 1 in \eqref{B1+a1+2b2} for $L=j$ is $v_{j,2}=0$.
Now let us consider \eqref{B1-a1} for $L=ijk$; the coefficient of $\eta_{ijk}$ is $t.v_{\emptyset,1}-4v_{\emptyset,1}=0$. Hence, by \eqref{tecres} we deduce $v_{\emptyset,1}=0$.\\
Moreover, let us consider \eqref{a0-B0} for $L=ijk$; the coefficient of $\eta_{ijk}$ is $-t.v_{\emptyset,0}+4v_{\emptyset,0}=0$. Hence by \eqref{tecres2} we deduce $v_{\emptyset,0}=0$.\\
Finally, let us consider \eqref{-a2+B2} for $L=ijk$; the coefficient of $\eta_{ijk}$ is $t.v_{\emptyset,2}-4v_{\emptyset,2}=0$. Hence  by \eqref{tecres3} we deduce $v_{\emptyset,2}=0$.\\
So far we have shown that, for all $i \in \left\{1,2,3,4,5,6\right\}$, $v_{\emptyset,0}=v_{\emptyset,1}=v_{\emptyset,2}=v_{i,1}=v_{i,2}=0$.\\
Let us now show that $v_{jl,1}=1$ for all $j l \in \mathcal{I}_{<}$.
The coefficient of $\eta_{l}$ in \eqref{B0+b1} for $L=j$ is $-\eta_{l} \otimes v_{jl,1}+\eta_{l} \otimes \xi_{lj} v_{\emptyset,0}=0.$ Therefore $v_{jl,1}=0$.

We know by \eqref{32ck6} that $bd_{2}=-4Bd_{1}+3ad_{1}$, since $Cd_{0}(\emptyset) =0$.
Using this relation we have that Equations \eqref{11ck6}, \eqref{21ck6}, \eqref{27ck6} reduce to:
\begin{align*}
C_{3}+4B_{4}+6a_{4}=2C_{3}+2(B_{4}-a_{4})-4iad_{1}+4iBd_{1}=C_{3}-6iad_{1}+6iBd_{1}=0.
\end{align*}
We consider the following linear combinations of the previous equations:
\begin{align*}
3B_{4}+7a_{4}+2iad_{1}-2iBd_{1}=B_{4}-a_{4}+4iad_{1}-4iBd_{1}=0.
\end{align*}
Since $ad_{1}(\emptyset)$ and $Bd_{1}(\emptyset)$ involve only terms in $\eta_{*}$ with $v_{\emptyset,1}$ that is 0, we obtain $a_{4}(\emptyset)=0$. Therefore
\begin{align*}
\sum_{I}\eta_{I}\otimes v_{I,4}=0.
\end{align*}
Using linear independence of distinct $\eta_{I}$'s, we get $v_{I,4}=0$ for all $I \in \mathcal{I}_{<}$.
Now Equations \eqref{12ck6}, \eqref{26ck6}, \eqref{31ck6} reduce to:
\begin{align*}
C_{2}+3a_{3}+3B_{3}=4C_{2}+3B_{3}-3a_{3}-3iad_{0}-3ibd_{1}=C_{2}   -6iBd_{0}  +3i ad_{0} -3ibd_{1}=0.
\end{align*}
We observe that $ad_{0}(\emptyset)$ and $Bd_{0}(\emptyset)$ involve only terms with $v_{\emptyset, 0}$ that is 0, $bd_{1}(\emptyset)$ involves only terms with $v_{\emptyset, 1},v_{I,1}$ where $|I|=1,2$, that are zero. Then these equations reduce to:
\begin{align*}
C_{2}+3a_{3}+3B_{3}=4C_{2}+3B_{3}-3a_{3}=C_{2}=0.
\end{align*}
Therefore $a_{3}(\emptyset)=0$. As before we deduce $v_{I,3}=0$ for all $I \in \mathcal{I}_{<}$.\\
Thus we have shown that, for a singular vector $\vec{m}$, $T(\vec{m})$ has the following form:
\begin{align*}
T(\vec{m})= \Theta^{2} \sum_{|I|\geq 2}\eta_{I} \otimes v_{I,2}+\Theta \sum_{|I|\geq 3}\eta_{I} \otimes v_{I,1}+ \sum_{|I|\geq 1}\eta_{I} \otimes v_{I,0}.
\end{align*}
This means that there are singular vectors $\vec{m}$ of at most degree 8 and, in particular, $T(\vec{m})$ has the following form:
\begin{align*}
&T(\vec{m})=\Theta^{2}\sum_{|I|=2}\eta_{I} \otimes v_{I,2} \quad \text{degree} \,\, 8,\\
&T(\vec{m})=\Theta^{2}\sum_{|I|=3}\eta_{I} \otimes v_{I,2} \quad  \text{degree} \quad 7,\\
&T(\vec{m})=\Theta^{2}\sum_{|I|=4}\eta_{I} \otimes v_{I,2} \quad  \text{degree} \,\, 6,\\
&T(\vec{m})=\Theta^{2}\sum_{|I|=5}\eta_{I} \otimes v_{I,2} +\Theta \sum_{|I|=3}\eta_{I} \otimes v_{I,1}+\sum_{|I|=1}\eta_{I} \otimes v_{I,0} \quad  \text{degree} \,\, 5,\\
&T(\vec{m})=\Theta^{2}\sum_{|I|=6}\eta_{I} \otimes v_{I,2} +\Theta \sum_{|I|=4}\eta_{I} \otimes v_{I,1}+\sum_{|I|=2}\eta_{I} \otimes v_{I,0} \quad  \text{degree} \,\, 4,\\
&T(\vec{m})=\Theta \sum_{|I|=5}\eta_{I} \otimes v_{I,1}+\sum_{|I|=3}\eta_{I} \otimes v_{I,0} \quad  \text{degree} \,\, 3,\\
&T(\vec{m})=\Theta \sum_{|I|=6}\eta_{I} \otimes v_{I,1}+\sum_{|I|=4}\eta_{I} \otimes v_{I,0} \quad  \text{degree} \,\, 2,\\
&T(\vec{m})= \sum_{|I|=5}\eta_{I} \otimes v_{I,0} \quad  \text{degree} \,\, 1.
\end{align*}
If we look respectively at vectors of degree $8$, $7$ and $6$, we can use relation $B_{1}(j)+a_{1}(j)+2b_{2}(j)=0$ from condition \textbf{S2} for $L=j$. In both these three cases it reduces to $b_{2}(j)=0$ since there are no $v_{I,1}$'s involved.
We get that:
\begin{align*}
b_{2}(j)=\sum_{I}\sgn_{I}\partial_{j}\eta_{I} \otimes v_{I,2} 
\end{align*}
where $\sgn_{I}=\pm 1$ and is not needed explicitly here, for $|I|=2,3,4$ respectively.
By linear independence we get $v_{I,2}=0$ for $ |I|=2,3,4$, $I \in \mathcal{I}_{<}$.
\end{proof}
\begin{ack}
	The author would like to thank Nicoletta Cantarini, Fabrizio Caselli and Victor Kac for useful comments and suggestions.
	\end{ack}


\begin{thebibliography}{-}
	\addcontentsline{toc}{chapter}{Bibliography}
	\bibitem{bagnoli} Bagnoli L., Caselli F. \textit{Classification of finite irreducible conformal modules for $K'_4$}, arXiv:2103.16374.
	\bibitem{kac1} Boyallian C., Kac V.G., Liberati, J. \textit{Irreducible modules over finite simple Lie conformal superalgebras of type K}, J. Math. Phys. 51 (2010), 1-37.
	\bibitem{ck6} Boyallian C., Kac V. G., Liberati J. \textit{Classification of finite irreducible modules over the Lie conformal superalgebra $CK_{6}$}, Comm. Math. Phys. 317 (2013), 503-546.
	\bibitem{bklr} Boyallian C., Kac  V. G., Liberati J., Rudakov A. \textit{Representations of simple finite Lie conformal superalgebras of type W and S}, J. Math. Phys. 47 (2006), 1-25.
	\bibitem{cantacasellikac} Cantarini N., Caselli F., Kac, V.G. \textit{Lie conformal superalgebras and duality of modules over linearly compact Lie superalgebras}, Adv. Math. 378 (2021), 107523.
	\bibitem{chengcantakac} Cheng S., Cantarini N., Kac V. G., \textit{Errata to Structure of Some $\Z$-graded Lie Superalgebras of Vector Fields}, Transf. Groups 9 (2004), 399-400.
	\bibitem{chengkac} Cheng S., Kac, V. G. \textit{Conformal modules}, Asian J. Math. 1, 181 (1997); 2, 153(E) (1998).
	\bibitem{new6} Cheng S., Kac V.G. \textit{A new $N = 6$ superconformal algebra}, Comm. Math. Phys. 186 (1997), 219-231.
	\bibitem{chengkac2} Cheng S., Kac, V. G. \textit{Structure of some $\Z$-graded Lie superalgebras of vector fields}, Transf. Groups, 4 (1999), 219-272.
	\bibitem{chenglam} Cheng, S., Lam, N. \textit{Finite conformal modules over N=2,3,4 superconformal algebras}, J. Math. Phys. 42 (2001), 906-933.
	\bibitem{dandrea} D'Andrea A, Kac V. G. \textit{Structure theory of finite conformal algebras}, Selecta Math., (N.S.) 4 (1998) 377-418.
	\bibitem{fattorikac} Fattori D., Kac V. G. \textit{Classification of finite simple Lie conformal superalgebras}, J. Algebra 258, (2002) 23-59., Special issue in celebration of Claudio Procesi's 60th birthday.
	\bibitem{kac1vertex} Kac V.G. \textit{Vertex algebras for beginners}, Univ. Lecture Ser., Vol. 10, AMS, Providence, RI, 1996, 2nd ed., (1998).
	\bibitem{kac98} Kac V.G. \textit{Classification of Infinite-Dimensional Simple Linearly Compact Lie Superalgebras}, Advanced in Mathematics 139 (1998), 11-55 
	\bibitem{kacrudakov} Kac V. G., Rudakov A. \textit{Representations of the exceptional Lie superalgebra E(3,6). I. Degeneracy conditions}, Transform. Groups 7, (2002) 67-86. 
	\bibitem{zm} Mart\'inez C., Zelmanov E., \textit{Irreducible representations of the exceptional Cheng-Kac superalgebra}, Trans. Amer. Math. Soc. 366 (2014), 5853-5876.
	\bibitem{Shche} Shchepochkina, I. \textit{The five exceptional simple Lie superalgebras of vector fields}, Funktsional Anal. i Prilozhen 33(3), 59-72, 96 (2000). transl. in Funct. Anal. Appl. 33 (1999), 3 208-219.
\end{thebibliography}
\end{document}